\documentclass[reqno]{amsart}
\usepackage{amsmath,amsthm,amssymb,amsfonts,amsfonts,amscd,mathrsfs}
\usepackage{url}
\usepackage{galois}
\usepackage{psfrag,graphicx}
\usepackage{epic,eepic}
\usepackage{setspace}

\thanks {2010 Mathematics Subjects Classification. Primary 37B40,
37A35, 22D25.}
\begin{document}
\theoremstyle{plain}
\newtheorem{main}{Theorem}
\renewcommand{\themain}{\Alph{main}}
\newtheorem{maincor}[main]{Corollary}
\newtheorem{Thm}{Theorem}[section]
\newtheorem{lem}[Thm]{Lemma}
\newtheorem{Prop}[Thm]{Proposition}
\newtheorem{Cor}[Thm]{Corollary}
\newtheorem{Cla}[Thm]{Claim}
\theoremstyle{remark}
\newtheorem{Def}[Thm] {Definition}
\newtheorem{Rem}[Thm] {Remark}
\newtheorem{Con}{Conjecture}
\newtheorem{Exa}[Thm] {Example}
\newtheorem{Que}[Thm] {Question}
\setlength{\parindent}{2em}
\newcommand{\var}{\ensuremath{\varepsilon}}
\newcommand{\und}{\ensuremath{\underline}}
\newcommand{\tbf}{\ensuremath{\textbf}}
\newcommand{\vhi}{\ensuremath{\varphi}}
\newcommand{\emp}{\ensuremath{\emptyset}}
\newcommand{\mcl}{\ensuremath{\mathcal}}
\renewcommand{\baselinestretch}{1.5}
\setlength{\parskip}{0\baselineskip}
\hspace*{\parindent}
\renewcommand{\baselinestretch}{1.8}
\addtocontents{toc}{\protect\vspace{5pt}}

\title
[Periodic measure]
{Periodic measures are dense in
invariant measures for residually finite
amenable group actions with specification}

\author[Ren]
{Xiankun Ren}

\maketitle
{\bf Abstract} We prove that for actions of a discrete countable residually
finite amenable group on a compact metric space with specification property,  periodic measures are dense in the
set of invariant measures. We also prove that some certain expansive
actions of a countable discrete group by automorphisms of compact abelian groups have specification property.

\section{Introduction}
Let $G$ be a discrete countable residually finite amenable group
acting on a compact metric space $X.$ Denote by $\mathcal{M}(X,G)$
the set of $G$-invariant measures and $\mathcal{M}_{e}(X,G)$ the
set of ergodic $G-$invariant measures. For a point $x\in X,$ we call
$x$ a periodic point if $|orb(x)|<\infty.$ Denote the periodic
measure $\mu_{x}$ as a probability measure with mass $|orb(x)|^{-1}$ at
each point of $orb(x)$ and we denote
$\mathcal{M}_{P}(X)$   the set of all such
periodic measures by $\mathcal{M}_{P}(X).$

For $G=\mathbb{Z}$ the theory of such dynamical system is well studied.
For example in \cite{Sigmund}, Sigmund proved the periodic measures are dense in
invariant measures for Axiom A-diffeomorphisms. In \cite{Hirayama}, Hirayama proved that the  measures supported by hyperbolic periodic points are dense in
invariant measures for  mixing $C^{1+\alpha}$ diffeomorphisms.
In \cite{Liang}, Liang, Liu and Sun proved that each invariant measure in a non-uniformly hyperbolic
system can be approximated by atomic measures on hyperbolic periodic
orbits.

Throughout this paper, $G$ will be a infinite discrete countable residually finite
amenable group acting on a compact metric space $X$ with metric $\rho.$ We will
recall some definitions and terminology in the next section.

Our main results are as follows.
\begin{Thm}\label{main1}
Let $G$  be a discrete countable residually finite
amenable group acting on a compact metric space $X$ with specification
property.
Then $\mcl{M}_{P}(X,G)$ is dense in $\mcl{M}(X,G)$ in the weak$^*$ topology.
Moreover $\mcl{M}_{e}(X,G) $ is residual in $\mcl{M}(X,G).$
\end{Thm}

\begin{Thm}\label{main2}
Let $\Gamma$ be a  countable discrete group and $f$ an element
of $\mathbb{Z}\Gamma$ invertible in $l^{1}(\Gamma, \mathbb{R}).$  Then the action of $\Gamma$ on $X_f$ has
specification property.

\end{Thm}

\section{Preliminary}
In this section we will set up some notions and basic facts
about amenable group and residually finite group. Also
we will give the definition of specification.

\subsection{Amenable group}

A countable group $G$ is  amenable if there exists a sequence of
finite subsets $\{F_n\}_{n\in\mathbb{N}}$ satisfying
$$\lim\limits_{n\rightarrow\infty}\frac{|F_n\triangle KF_n|}{|F_n|}=0, \forall K\in \mcl{F}(G),$$
where $\mcl{F}(G)$ is the collection of all finite subsets of $G.$
Such sequences are said F{\o}lner sequences.

The quasi tiling theory is a useful tool for amenable group actions which is set up by Ornstein and Weiss in \cite{Orn}. Subsets $A_1, A_2,\cdots, A_k\in \mcl{F}(G)$ are $\var$-disjoint if there exists $\{B_1,B_2,\cdots,B_k\}\subset \mcl{F}(G)$ such that
\\$(1) B_i\subset A_i \qquad i=1,2,\cdots,k,$
\\$(2) \frac{|B_i|}{|A_i|}>1-\var\ \quad i=1,2,\cdots,k.$

\noindent For $\alpha\in (0,1],$ we say $\{A_1,A_2,\cdots,A_k\}$ $\alpha$-covers $A\in\mcl{F}(G)$ if
$$\frac{A\cap(\bigcup_{i=1}^{k}A_i)}{|A|}\geq \alpha.$$

\noindent We say that $\{A_1,A_2,\cdots,A_k\}\subset \mcl{F}(G)$ $\var-$quasi-tile $A\in \mcl{F}(G)$ if there exists $\{C_1,C_2,\cdots,C_k\}\subset \mcl{F}(G)$ satisfying
\\(1)$A_iC_i\subset A $ and $\{A_ic|c\in C_i\}$ forms an $\var-$disjoint family for $i=1,2,\cdots,k,$
\\(2)$A_iC_i\cap A_jC_j\neq \emptyset \quad 1\leq i\neq j\leq k,$
\\(3)$\{A_iC_i: i=1,2,\cdots k\}$ forms a $(1-\var)-$cover of $A.$
The subsets $C_1,C_2,\cdots,C_k$ are called the tiling centers.

The following proposition is  \cite[Theorem 2.6]{Ward}.
\begin{Prop}
Let $\{F_n\}_{n\in\mathbb{N}}$ with $\{e_G\}\in F_1\subset F_2\subset$ and $\{F_n^{\prime}\}_{n\in\mathbb{N}}$ be two F{\o}lner sequences of $G.$ Then for any $\var\in(0,\frac{1}{4})$ and $N\in\mathbb{N}$ there exists integers $n_1,n_2,\cdots,n_k$ with $n_k>n_{k-1}>\cdots>n_1>N $ such that $F_{n_1},F_{n_2},\cdots,F_{n_k}$ $\var-$quasi-tile $F_m^{\prime}$ when $m$ is large enough.
\end{Prop}\label{tile}

We will change  Proposition \ref{tile} a little bit for our use. This
is a light reformulation of \cite[Proposition 2.3]{Huang}. We describe these changes
briefly, using the same notation and terminology in \cite{Huang}.
Fix $\var\in (0,\frac{1}{4}),N\in\mathbb{N}$ and $k,\delta$ are as in \cite[Proposition 2.3]{Huang}.  Choose $n_1,n_2,\cdots,n_k$ such that
$F_{n_{i+1}}$ is $(F_{n_i}F_{n_i}^{-1},\delta)-$invariant and
$\frac{|F_{n_{i}}|}{|F_{n_{i+1}}|}<\delta, \quad i=1,2,\cdots,k-1.$
Then there is some $g_i\in G,$ satisfying $F_{n_{i}}g_i\subset
F_{n_{i+1}}.$ Denote by $h_i=g_ig_{i+1}\dots g_{k-1},\quad \text{for} \quad i=1,2\cdots,k-1$ and $h_k=e_G,$ then
$F_{n_1}h_1\subset F_{n_2}h_2 \subset\cdots\subset F_{n_k}.$
Pich $g\in G$ such that $e_G\in F_{n_1}h_1g.$ Denote by
$\widetilde{{F}}_{n_i}=F_{n_i}h_ig.$ Then $e_G\in
\widetilde{{F}}_{n_1}\subset \widetilde{{F}}_{n_2}\subset\cdots
\widetilde{{F}}_{n_k}.$ From proposition 2.3 in \cite{Huang}, we know $F_m^{\prime}$ can be $\var-$quasi tiled by $\widetilde{{F}}_{n_1},\widetilde{{F}}_{n_2},\cdots
\widetilde{{F}}_{n_k}$ when $m$ is sufficiently large. Since $\widetilde{{F}}_{n_i}$ is just a translation
of $F_{n_i}$ for $i=1,2,\cdots,k,$ so $F_m^{\prime}$ can be $\var-$quasi-tiled
by ${{F}}_{n_1},{{F}}_{n_2},\cdots,{{F}}_{n_k}.$
Hence we get the next proposition.
\begin{Prop}\label{prime}
Let $\{F_n\}_{n\in\mathbb{N}}$ and $\{F_n^{\prime}\}_{n\in\mathbb{N}}$ be two F{\o}lner sequences of $G.$ Then for any $\var\in(0,\frac{1}{4})$ and $N\in\mathbb{N}$ there exists integers $n_1,n_2,\cdots,n_k$ with $n_k>n_{k-1}>\cdots>n_1>N $ such that $F_{n_1},F_{n_2},\cdots,F_{n_k}$ $\var-$quasi-tile $F_m^{\prime}$ when $m$ is large enough.
\end{Prop}

For our proof, we also need the Mean Ergodic Theorem for
amenable group actions . Readers may refer \cite[Theorem 1.1]{Lin1}.

\begin{lem}[Mean Ergodic Theorem]
  Let $G$ be an amenable group acting on a measure space $(X,\mathcal{B},\mu)$ by measure preserving transformation, and let $\{F_n\}_{n\in \mathbb{N}}$ be a  F{\o}lner sequence. Then for any $f \in L^{1}(\mu)$, there is a $G$-invariant $\overline{f}\in L^{1}(\mu)$ such that $$\lim\limits_{n\rightarrow\infty}\frac{1}{|F_n|}\sum\limits_{g\in F_n}f(gx)=\overline{f}(x)\quad \text{in}\quad L^{1}.$$
  Moreover $\int f(x) d\mu=\int \overline{f}(x) d\mu.$
 In particular, if the $G$-action is ergodic,
 $$\lim\limits_{n\rightarrow\infty}\frac{1}{|F_n|}\sum\limits_{g\in F_n}f(gx)=\int f(x){\rm d}\mu(x) \quad \text{in}\quad L^{1}.$$ \label{ergodic theorem}
 \end{lem}
For more information about ergodic theorem, readers may refer
\cite{Lin} or \cite[Chapter 8]{Ein}.

\subsection{Residually finite group}

A group is residually finite if the intersection of all its normal subgroups of finite index is trivial. Examples of groups that are residually finite are finite groups, free groups, finitely generated nilpotent groups, polycyclic-by-finite groups, finitely generated linear groups, and fundamental groups of 3-manifolds. For more information
about residually finite groups readers can refer \cite[Chapter 2]{Cec}.

Let $(G_n,n\geq 1)$ be a sequence of finite index normal subgroups in $G.$ We say
$$\lim_{n\rightarrow \infty} G_n=\{e_G\}$$
if we can find, for any $K\in\mcl{F}(G),$ an $N>1$ with $G_n\cap(K^{-1}K)=\{e_G\}$ for every $n\geq N.$ Clearly such sequence exists if $G$ is countable and residually finite.

If $G^{\prime}\subset G$ is a subgroup with finite index,
and $Q\subset G$ is a {\it fundamental domain} of the right coset
space $G/G^{\prime},$ i.e. a finite subset such that
$\{sQ\,|\,s\in G^{\prime}\}$ is a partition of $G.$

The following proposition is \cite[Corollary 5.6]{Den} and we will use
it to control the periodic orbits we get.
\begin{Prop}\label{finite}
Let $G$ be a countable discrete residually finite amenable group and let $(G_n,n\geq 1)$ be a sequence of finite index normal subgroup with $\lim_{n\rightarrow \infty}G_n=\{e_G\}.$ Then there exists a F{\o}lner sequence $(Q_n,n\geq 1)$ such that $Q_n$ is a fundamental domain of $G/G_n$ for every $n\geq 1.$

\subsection{Specification}

Specification is an orbit tracing property which is very useful in finding periodic orbits. In \cite{Bowen}, Bowen was the first to state this notion and Ruelle investigated the extension of this notion to $\mathbb{Z}^{d}$-action in \cite{Rue}. Chung and Li again generalized the property to general group actions in \cite{Chung}.

Let $\alpha$ be a continuous $G-$action on a compact metric space $X$ with metric $\rho.$ The action has specification property if there exist, for every $\varepsilon>0,$ a nonempty finite subset $F=F(\var)$ of $G$ with the following property :
\\for any finite collection of finite subsets $F_1, F_2,\cdots, F_m$ of $G$ satisfying
\begin{align}
FF_i\cap F_j=\emptyset \qquad 1\leq i\neq j\leq m \label{specification1}
\end{align}
and for any subgroup $G^{\prime}$ of $G$ with
\begin{align}
FF_i\cap F_j(G^{\prime}\setminus \{e_G\})=\emptyset \quad\text{for}\quad 1\leq i, j\leq m. \label{specification2}
\end{align}
Then for any collection of points $x^1,x^2,\dots,x^m \in X$
there is a point $y\in X$ satisfying
\begin{align}
\rho(sx^{i},sy)\leq \var \quad \text{for all}\quad s\in F_i,\ 1\leq i\leq m \label{con1}
\end{align}
and $sy=y$ for all $s\in G^{\prime}.$

\section{Proof of Theorem}
Let $(G_n,n\geq 1)$ be a sequence of finite index normal subgroup
with $\lim\limits_{n\rightarrow\infty}G_n=\{e_G\}$
and a F{\o}lner sequence $(Q_n, n\geq 1)$ such that $Q_n$ is a fundamental domain of $G/G_n$ as described in Proposition \ref{finite}.

Let $\nu\in \mcl{M}(X,G), \var >0$ and $W$ a finite subset of $C(X).$ Uniformly continuity of the elements of $W$ implies that there is $\delta\in(0,\var)$ such that $|\xi(x)-\xi(y)|<\frac{\var}{8}$ for all $x,y\in X$ with $d(x,y)<\delta$ for all $\xi\in W.$

Define
$$Q(G)=\{x\in X: \lim\limits_{n\rightarrow\infty}\frac{1}{|Q_n|}\sum\limits
_{g\in Q_n}\xi(gx) \quad\text{exists for all}\quad \xi\in C(X)\}$$

We call $Q(G)$ the basin of $\nu.$ By Birkhorff Ergodic theorem, $\nu(Q(G))=1.$ Denote by $\xi^{*}(x)$  the limit for each $x\in Q(G).$

Moreover,
$$\int_{Q(G)}\xi^{*}(x)d\nu=
\int_{Q(G)}\xi(x)d\nu.$$

\noindent Next we will construct a finite partition as following:

\noindent Set $D=\max\limits_{\xi\in W} \|\xi\|_{\infty}$

\noindent For $j=1,2,\cdots,[\frac{16D}{\var}]+1$
define
$$Q_j(\xi)=\{x\in Q(G)|
-D+\frac{j-1}{8}\var\leq \xi^{*}(x)<-D+\frac{j}{8}\var\}.$$

\noindent Since $W$ is finite
$$\eta:=\bigvee\limits_{\xi\in W}\{Q_1(\xi),\cdots,
Q_{[\frac{8D}{\var}]+1}(\xi)\}$$
is a finite partition of $Q(G),$ where
 $\alpha\vee\beta=\{A_i\cap B_j|A_i\in\alpha, B_j\in \beta\}$ for partitions $\alpha=\{A_i\}$ and $\beta=\{B_j\}$

Next we will construct $F_1,F_2,\cdots,F_t$ and $G_m$ satisfying (\ref{specification1}) and (\ref{specification2}) in specification property. The idea is from \cite{Zheng} but some little modifications. Suppose
$\eta=\{A_1,A_2,\cdots,A_l\}.$ Let $a_i=\nu(A_i)$ for $i=1,2,\cdots,l$
and $a=\min{\{a_i:i=1,2,\cdots,l\}}.$ By Egorov's Theorem, there exist a Borel subset
$X^{\prime}\subset X$ with $\nu(X^{\prime})>1-\frac{1}{4}a$ and
$\frac{1}{|Q_n|}\sum\limits_{g\in Q_n}\xi(gx)$ converges to $\xi^{*}(x)$
uniformly on $X^{\prime}.$

Take $0<\gamma<\min{\{\frac{\delta}{16l},\frac{\delta}{8D|F|},\frac{1}{2},\frac{\delta}{16Dl}\}}$ and $F=F(\gamma)$ as in the specification property. Take $N_1\in\mathbb{N}$ such that for $\forall \ x\in X^{\prime},$
 $|\frac{1}{|Q_n|}\sum\limits_{g\in Q_n}\xi(gx)-\xi^{*}(x)|<\frac{1}{8}\var,$ for all $n\geq N_1.$
   Take $N_2>N_1$ large enough s.t. $\frac{|gQ_n\triangle Q_n|}{|Q_n|}<\frac{\gamma}{4|F|^{2}l}$ for all $g\in F,\quad n\geq N_2.$
   By Proposition \ref{prime}
 there exist $n_k>n_{k-1}>\cdots>n_1>N_2$ and $N_3\in \mathbb{N}$ s.t. $Q_m$ can be $\frac{\gamma}{4|F|^2l}-$quasi-tile by $Q_{n_1},Q_{n_2},\cdots,Q_{n_k}$ when $m>N_3.$ Also $N_3$ will be large enough such that the family of all the translations
 $$\mcl{F}=\{Q_{n_j}c_j: 1\leq i\leq k, c_j\in C_j\}$$
 can be partitioned into $l$ subfamilies $\mcl{F_i},\quad 1\leq i\leq l$ satisfying
 $$|\frac{|\mcl{F}_i|}{|Q_m|}-a_i|<
 \frac{\gamma}{l}$$
where \quad $\bigcup\mcl{F}_i=\bigcup \limits_{Q_{n_j}c_j\in \mcl{F}_i}\{F_{n_j}c_j\}.$
Moreover the elements in $\mcl{F}$ are
pairwise $\frac{\gamma}{4|F|^2l}-$disjoint. We can choose
$\{\widetilde{T_{n_j}}(c_j)c_{j}\subset Q_{n_j}c_j\in \mcl{F}\}$ are pairwise disjoint and $\frac{|\widetilde{T_{n_j}}(c_j)|}{|F_{n_j}|}>1-\frac{\gamma}{4|F|^2l}.$
Denote $$T_{n_j}(c_j)=\{s\in \widetilde{T}_{n_j}(c_j)|Fs\subset Q_{n_j}\}$$
and
$$S_{n_j}(c_j)=T_{n_j}(c_j)\cap FT_{n_j}(c_j).$$
By the definition, we know
$\frac{|Q_{n_j}\setminus T_{n_j}(c_j)|}{|Q_{n_j}|}<\frac{\gamma}{2|F|}$
and $\frac{|Q_{n_j}\setminus S_{n_j}(c_j)|}{|Q_{n_j}|}<\gamma.$
Denote $\widetilde{\mcl{F}}=\{S_{n_j}(c_j)c_j|S_{n_j}(c_j)c_j\subset Q_{n_j}c_j\in \widetilde{F}\}$
and $\widetilde{\mcl{F}}_i=\{S_{n_j}(c_j)c_j|Q_{n_j}c_j \in \mcl{F}_i\}$

\noindent{\bf Claim:} $\{S_{n_j}(c_j)c_j|S_{n_j}(c_j)c_j\in \widetilde{\mcl{F}}\}$ and $G_m$ satisfy
the conditions in specification property.

\noindent For different $S_{n_{j_1}}(c_{j_1})c_{j_1}$ and $S_{n_{j_2}}(c_{j_2})c_{j_2}$
we have
$$(S_{n_{j_1}}(c_{j_1})c_{j_1}\cap FS_{n_{j_2}}(c_{j_2})c_{j_2})\subset
(S_{n_{j_1}}(c_{j_1})c_{j_1}\cap T_{n_{j_2}}(c_{j_2})c_{j_2})\subset
(\widetilde{T}_{n_{j_1}}(c_{j_1})c_{j_1}\cap \widetilde{T}_{n_{j_2}}(c_{j_2})c_{j_2}
)=\emptyset
$$

\noindent By Proposition \ref{finite}, we know
\begin{align}
G
&= \bigsqcup\limits_{g\in Q_n}gG_{n} \label{spec1}\\
&= \bigsqcup\limits_{g\in Q_n}(g\bigsqcup g(G_{n}\setminus \{e_G\}))\label{spec2}
\end{align}
For different $S_{n_{j_1}}(c_{j_1})c_{j_1}$ and $S_{n_{j_2}}(c_{j_2})c_{j_2}$,
since $FS_{n_{j_1}}(c_{j_1})c_{j_1}\cap S_{n_{j_2}}(c_{j_2})c_{j_2}=\emptyset$,
by (\ref{spec1}), $FS_{n_{j_1}}(c_{j_1})c_{j_1}\cap S_{n_{j_2}}(c_{j_2})c_{j_2}G_m = \emptyset$;
for the same $S_{n_{j_1}}(c_{j_1})c_{j_1}$ and $S_{n_{j_2}}(c_{j_2})c_{j_2}$,
by (\ref{spec2}), we get $FS_{n_{j_1}}(c_{j_1})c_{j_1}\cap S_{n_{j_2}}(c_{j_2})c_{j_2}(G_m\setminus{\{e_G\}}) = \emptyset.$
We get that $\{S_{n_j}(c_j)c_j\}$ and $G_m$
satisfy the conditions (\ref{specification1}) and (\ref{specification2}) in the specification property.

Next we will construct the periodic measure.

For $S_{n_j}(c_j)c_j\in \widetilde{\mcl{F}_i},$ pick some $x$ related to
$S_{n_j}(c_j)c_j$ denoted by $x_j(c_j)$ satisfying
$x_j(c_j)\in c_j^{-1}A_j\cap X^{\prime}.$ Using the specification property,
there is some $y$ such that $\rho(gx_j(c_j),gy)<\gamma<\delta, \forall g\in S_{n_j}(c_j)c_j.$ Denote
$$\mu_y=\frac{1}{|\widetilde{\mcl{F}}|}
(\sum\limits_{i=1}^{l}\sum\limits_{S_{n_j}(c_j)c_j\in\cup\widetilde{\mcl{F}}_i}
\,\sum\limits_{g\in S_{n_j}(c_j)c_j}\delta_{gy})$$
where $\delta_y$ is the Dirac measure on $y.$

\noindent {\bf Claim:} $|\int\xi d\nu-\int\xi d\mu_y|<\var,\quad \quad\text{for all} \quad\xi\in W$
\begin{align}
|\int\xi d\nu-\int\xi d\mu_y|
&=|\int\xi^{*}(x) d\nu-\int\xi(x) d\mu_{y}| \nonumber\\
&= |\int\xi^{*}(x) d\nu-\frac{1}{|\cup\widetilde{\mcl{F}}|}
\sum\limits_{i=1}^{l}\sum\limits_{S_{n_j}(c_j)c_j\in \widetilde{\mcl{F}}_i}
\sum\limits_{g\in S_{n_j}(c_j)c_j}\xi(gy)| \nonumber\\
&\leq |\int\xi^{*}(x) d\nu-\frac{1}{|\cup\widetilde{\mcl{F}}|}
\sum\limits_{i=1}^{l}\sum\limits_{S_{n_j}(c_j)c_j\in \widetilde{\mcl{F}}_i}
\sum\limits_{g\in S_{n_j}(c_j)c_j}\xi(gx_{n_j}(c_j))|+\frac{\var}{8}\nonumber\\
&\leq |\int\xi^{*}(x) d\nu-\sum\limits_{i=1}^{l}\sum\limits_{S_{n_j}(c_j)c_j\in \widetilde{\mcl{F}}_i}
\frac{|S_{n_j}(c_j)|}{|\cup\widetilde{\mcl{F}}|}
\frac{1}{|S_{n_j}(c_j)|}\sum\limits_{g\in S_{n_j}(c_j)c_j}\xi( gx_j(c_j))|+\frac{\var}{8}\label{e1}
\end{align}

Pick $x_i\in A_i\cap X^{*},$ then $|\xi^*(x)-\xi^*(x_i)|<\frac{\var}{8}$ for
all $x\in A_i$
by the construction of $\eta.$
Then
\begin{align}
|\frac{1}{|Q_{n_j}|}\sum\limits_{g\in Q_{n_j}}\xi(gc_jx_j(c_j))-
\frac{1}{|Q_{n_j}|}\sum\limits_{g\in Q_{n_j}}\xi(gx_i)|
&\leq |\xi^{*}(c_jx_j(c_j))-\xi^{*}(x_i)|+\frac{\var}{4}\nonumber\\
&\leq \frac{3}{8}\var \label{no}
\end{align}
Also
\begin{align}
&|\frac{1}{|S_{n_j}(c_j)|}\sum\limits_{g\in S_{n_j}}\xi(gc_jx_j(c_j))- \frac{1}{|Q_{n_j}|}\sum\limits_{g\in Q_{n_j}}\xi(gc_jx_j(c_j))| \nonumber\\
&\leq |\frac{1}{|S_{n_j}(c_j)|}\sum\limits_{g\in S_{n_j}(c_j)}\xi(gc_jx_j(c_j))
-\frac{1}{|S_{n_j}(c_j)|}\sum\limits_{g\in Q_{n_j}}\xi(gc_jx_j(c_j))|\nonumber\\
&+ |\frac{1}{|S_{n_j}(c_j)|}\sum\limits_{g\in Q_{n_j}}\xi(gc_jx_j(c_j))
-\frac{1}{|Q_{n_j}|}\sum\limits_{g\in Q_{n_j}}\xi(gc_jx_j(c_j))| \nonumber\\
&\leq \frac{1}{|S_{n_j}(c_j)|}D|Q_{n_j}\setminus S_{n_j}(c_j)|
+\frac{|Q_{n_j}|D}{|S_{n_j}(c_j)||Q_{n_j}|}(|Q_{n_j}|-|S_{n_j}(c_j)|)\nonumber\\
&\leq \frac{2D\gamma}{1-\gamma}\nonumber\\
&\leq \frac{1}{4}\var\label{e2}
\end{align}

Taking (\ref{no}) and (\ref{e2}) into (\ref{e1}), we have,
\begin{align}
|\int\xi d\nu-\int\xi d\mu_y|
&\leq |\int\xi^{*} d\nu-\sum\limits_{i=1}^{l}\sum\limits_{S_{n_j}c_j\in\widetilde{\mcl{F}}_i}
\frac{|S_{n_j}(c_j)|}{|\cup \widetilde{\mcl{F}}|}\frac{1}{|Q_{n_j}|}\sum\limits_{g\in Q_{n_j}}
\xi(gc_jx_j(c_j))|+\frac{3}{8}\var \nonumber\\
&\leq |\int\xi^{*}(x) d\nu-\sum\limits_{i=1}^{l}\sum\limits_{S_{n_j}(c_j)c_j\in\widetilde{\mcl{F}}_i}
\frac{|S_{n_j}(c_j)|}{|\cup\widetilde{\widetilde{F}}|}\xi^{*}(c_jx_j(c_j))|+\frac{1}{2}\var\nonumber\\
&\leq |\int\xi^{*}(x) d\nu-\sum\limits_{i=1}^{l}a_i\xi^{*}(x_i)|+\frac{3}{4}\var\nonumber\\
&\leq |\sum\limits_{i=1}^{l}\int_{A_i}\xi^{*}(x) -\xi^{*}(x_i) d\nu|+\frac{7}{8}\var\nonumber\\
&\leq \var
\end{align}

For the moreover part, we just need the following claim.\\
{\bf Claim:} Suppose $(K,\rho)$ is a convex compact metric set and
denote the set of extreme points of $K$ by $ext(K).$ Then $ext(K)$
is a $G_\delta$ subset of $K.$\\
Let
$$K_n=\{x\in K: \ \text{there exist}\  y,z\in K\ \text{such that}\ x=\frac{1}{2}(y+z) \ \text{and} \ d(y,z)\geq\frac{1}{n}\}.$$
Obviously, $K_n$ is closed. Let $K_0=\bigcup\limits_{n\geq 1}K_n$
and $K_0$ is a $F_\sigma$ subset.
It is easy to check
$$x\notin ext(K) \Leftrightarrow x\in K_0.$$
As a result, $ext(K)=K\setminus K_0$ is a $G_\delta$ subset of
$K.$
We know that $\mcl{M}_e(X,G)$ is the set of extreme points of
$\mcl{M}(X,G).$ So by the claim,  $\mcl{M}_e(X,G)$ is a $G_\delta$
set. Combined with periodic measure is ergodic, we get the moreover
part.

So we prove the Theorem 1.1.\qed
\end{Prop}

\section{Proof of Theorem 1.2}
For a countable group $\Gamma$ and an element  $f=\sum f_s s$ in the
integral group ring $\mathbb{Z}\Gamma,$ consider the quotient
$\mathbb{Z}\Gamma/\mathbb{Z}\Gamma f$ of $\mathbb{Z}\Gamma$ by
the left ideal $\mathbb{Z}\Gamma f$ generated by $f.$ It is a discrete
abelian group with a left $\Gamma$-action by multiplication. The
Pontryagin dual
$$X_f=\widehat{\mathbb{Z}\Gamma/\mathbb{Z}\Gamma f}$$
is a compact abelian group with a left action of $\Gamma$ by
continuous group antomorphisms and we denote by $\rho$ the compliable metric on $X_f.$
Denote $X=(\mathbb{R}/\mathbb{Z})^{\Gamma}.$ The left and right
actions $l$ and $r$ on $X$ are defined by
$$(l^{s}x)_t=x_{s^{-1}t} \ \text{and}\ (r^{s}x)_t=x_{ts}$$
for every $s,t\in \Gamma$ and $x\in X.$ We can extend those actions
of $\Gamma$ to commuting actions $l$ and $r$ of $\mathbb{Z}\Gamma$
on $X$ by setting
$$l_{f}x=\sum\limits_{s\in \Gamma}f_s l^{s}x \
\text{and}\ r_{f}x=\sum\limits_{s\in\Gamma}f_s r^{s}x.$$
It is easy to check $X_f=\{x\in X\,|\, r_{f}x=xr^{*}=0\},$ where
$f^{*}=\sum f_{s}s^{-1}.$ Denote by
$$\alpha_f=l|_{X_f},$$
the restriction to $X_f$ of the $\Gamma$-action $l$ on $X.$

From in\cite[Theorem 3.2]{Den} , {\bf the action $\alpha_{f}$ is
expansive if and only if $f$ is invertible in $l^{1}(\Gamma).$}
Denote by $P$ the canonical map $l^{\infty}(\Gamma, \mathbb{R})\rightarrow X.$
Let $\xi=P\comp r_{f^{-1}}: L^{\infty}(\Gamma, \mathbb{Z}) \rightarrow
X_f.$ By in \cite[Proposition 4.2]{Den}, $\xi$ is a surjective
group homomorphism.

\noindent The following lemma is in \cite[Lemma 4.5]{Den}
\begin{lem}\label{small}
For every $x\in X_f$ there exists an element $v\in l^{\infty}(\Gamma,\mathbb{Z})$ with $\xi(v)=x$
and $\| v\|_{\infty}\leq \frac{\| f\|_1}{2}.$
\end{lem}
\begin{proof}
The proof is really easy. We choose $w\in P^{-1}(X_f)\subset
l^{\infty}(\Gamma, \mathbb{R})$ with $P(w)=x$ and
$-\frac{1}{2}\leq w_s<\frac{1}{2}$ for every $s\in \Gamma.$
Then $v=r_{f}(v)$ is what we need.
\end{proof}
Denote by $W=\{e_\Gamma\}\cup support(f^{*})=(\{e_\Gamma\}\cup support(f))^{-1}.$
Let $\var>0.$ Then we can fine a nonempty finite subset $W_1$
of $\Gamma$ and $\var_1\in (0,\| f\|_1^{-1})$
such that $x,y\in X_f$ satisfy $\max_{s\in W_1}|x_s-y_s|\leq 2\var_1,$
then $\rho(x,y)\leq \var.$ Take a finite subset $W_2$ of $\Gamma$
containing $e_\Gamma$ satisfying $\sum_{s\in\Gamma\setminus W_2^{-1}}\| (f^{*})^{-1}\|_{1}<\frac{\var_1}{2\| f\|_1}.$ Denote by $\tilde{F}=W_1W_2(W_1W_2)^{-1}.$

For any finite collection of finite subsets $F_1, F_2,\cdots, F_m$ of $\Gamma$ satisfying
$\tilde{F}F_i\cap F_j=\emptyset, 1\leq i\neq j\leq m$ and
any collection of points $x^1, x^2, \dots, x^{m}.$ By Lemma \ref{small}, pick
$v^1, v^2, \dots, v^m \in l^{\infty}(\Gamma, \mathbb{Z})$
with $\parallel v^i\parallel_{\infty}\leq \frac{\parallel f\parallel_1}{2}$
and $\xi(v^i)=x^i$ for $i=1,2,\dots, m.$ Let $v\in l^{\infty}(\Gamma,\mathbb{Z})$ be a point with
$\parallel v\parallel_{\infty}\leq \frac{\parallel f\parallel_1}{2}$
and $v_s=v^i_s$ for $s\in F_i^{-1}W_1W_2.$ $v$ is well defined since
$F_i^{-1}W_1W_2\cap F_j^{-1}W_1W_2=\emptyset$ for $1\leq i\neq j\leq m.$

Let
$$\tilde{w}_s=
\begin{cases}
(f^{-1})_s & s\in W_2\\
0& otherwise.
\end{cases}$$
{\bf Claim1:} $y=\xi(v)$ satisfies
$\rho(sx^i,sy)<\var $ $\forall s\in F_i, \ i=1,2,\dots, m.$

\noindent By the choice of $W_1,$ we just need to show
$|(sx^i)_t-(sy)_{t}|<2\var_1$ $\forall s\in F_i$ and $t\in W_1$
i.e. $|x^i_{s}-y_s|<2\var_1$ for all $s\in F_i^{-1}W_1.$

\noindent For any $s\in F_i^{-1}W_1,$
\begin{align}
|x^i_{s}-y_s| &= |P\comp r_{f^{-1}}(v^{i})_s-P\comp r_{f^{-1}}(v)_s|\nonumber\\
&= |P\comp r_{f^{-1}}(v^{i})_s-P\comp r_{\tilde {w}}(v^{i})_s
+P\comp r_{\tilde {w}}(v^{i})_s-P\comp r_{\tilde {w}}(v)_s\nonumber\\
&+P\comp r_{\tilde {w}}(v)_s-P\comp r_{f^{-1}}(v^{i})_s|.\nonumber
\end{align}
For $s\in F_i^{-1}W_1,$ we get the following calculations,
\begin{align}
|r_{f^{-1}}(v^{i})_s- r_{\tilde {w}}(v^{i})_s|
&=|\sum_{t\in\Gamma}f^{-1}_tv^{i}_{st}-\tilde{w}_{t}v^{i}_{st}|\nonumber\\
&\leq |\sum_{t\in W_2}f^{-1}_tv^{i}_{st}-\tilde{w}_{t}v^{i}_{st}|+
|\sum_{t\in\Gamma\setminus W_2}f_tv^{i}_{st}-\tilde{w}_{t}v^{i}_{st}|\nonumber\\
&\leq \frac{\var_1}{2\|f\|_1}\cdot \frac{\|f\|_1}{2}=\frac{\var_1}{4}\label{e3}
\end{align}
and
\begin{align}
|r_{\tilde {w}}(v^{i})_s- r_{\tilde {w}}(v)_s|
&=|\sum_{t\in\Gamma}\tilde {w}_t(v^{i})_{st}- \tilde {w}_tv_{st}|\nonumber\\
&=|\sum_{t\in W_2}\tilde {w}_t(v^{i})_{st}- \tilde w_tv_{st}|
=0
\end{align}
and
\begin{align}
|r_{\tilde {w}}(v)_s- r_{f^{-1}}(v)_s|
&=|\sum_{t\in\Gamma}\tilde{w}_tv_{st}-f^{-1}_tv_{st}|\nonumber\\
&\leq |\sum_{t\in W_2}\tilde{w}_tv_{st}-f^{-1}_tv_{st}|
 +|\sum_{t\in\Gamma\setminus W_2}\tilde{w}_tv_{st}-f^{-1}_tv_{st}|\nonumber\\
&\leq |\sum_{t\in\Gamma\setminus W_2}f^{-1}_tv_{st}| \nonumber\\
&\leq \frac{\var_1}{4}\label{e4}
\end{align}
By (\ref{e3})--(\ref{e4}), we get
$$\rho(sx^i,sy)< \var \ \text{for all}\ s\in F_i, \ i=1,2,\dots, m.$$

The following lemma is  a version  \cite[Lemma 1]{Bryant}, also  appeared in the proof of Lemma 6.2
in \cite{Chung}.

\begin{lem}\label{exp1}
Let $\alpha$ be an expansive continuous action of $\Gamma$
on a compact metric space $(X,\rho).$ Let $d>0$ such that
if $x,y\in X$ satisfy $\sup_{s\in \Gamma}\rho(sx,sy)\leq d,$
then $x=y.$ Let $x,y$ satisfy $\rho(sx,sy)\leq d$ for all but finite
many $s\in \Gamma.$ Then $\rho(sx,sy) \rightarrow 0$ as
$\Gamma\ni s \rightarrow \infty.$
\end{lem}

\noindent Denote $\Delta(X_f)=\{x\in X_f | \rho(sx,se_{X_f}) \rightarrow 0
\ \text{as} \ \Gamma \ni s \rightarrow \infty\}.$

\noindent {\bf Claim2:} Let $d$ be the the expansive coefficient of $(X_f, \alpha_f)$ i.e. for $x,y\in X_f$ with $\rho(sx,sy)\leq d$ for all $s\in \Gamma$ implies $x=y.$   For any  $\var \in (0,d)$
  and $\tilde{F}(\var)=\tilde{F}$ as in
Claim1. Then for any finite subset $F_1$ of $\Gamma$
and $x\in X_f,$ there exists $y\in \Delta(X_f)$
s.t. $\max_{s\in F_1}\rho(sx,sy)\leq \var$ and
$\sup_{s\in \Gamma\setminus \tilde{F}F_1}\rho(e_{X_f},sy)\leq \var.$

To prove the above claim, we may assume $\tilde{F}=\tilde{F}^{-1}$ otherwise
we can replace $\tilde{F}$ by $\tilde{F}\cup\tilde{F}^{-1}.$ Let
$F_1$ be a finite subset of $\Gamma$ and $x\in X_f.$ For
each finite set $F_2\subset \Gamma\setminus \tilde{F}F_1.$
From claim1, we can find $y_{F_2}\in X_f$ such that
$\rho(sx,sy_{F_2})\leq \var$ for all $s\in F_1$ and
$\rho(se_{X_f},sy_{F_2})\leq \var$ for all $s\in F_2.$ Note that the collection of
the finite subsets of $\Gamma\setminus \tilde{F}F_1$ has a partial
order. Take a limit point $y\in X_f$ of $\{y_{F_2}\}_{F_2}.$
Then $\rho(sx,sy)\leq \var$ for all $s\in F_1$
and $\rho(se_{X_f},sy)\leq \var$ for all $\Gamma\setminus FF_1.$ By Lemma \ref{exp1}, we
know $y\in \Delta(X_f).$ \qed

Take some $\delta>0$ satisfying $\rho(x,y)<\delta$
implies $|x_{e_\Gamma}-y_{e_\Gamma}|<\var_1,$ for
all $x,y\in X_f.$ Take $W_3=W_3(\delta)$ as described in Claim2,
such that for any finite subsets $F_1$ of $\Gamma$
and $x\in X_f$ there exists $y\in \Delta(X_f)$ with
$\rho(sx,sy)<\delta$ for all $s\in F_1$ and
$\rho(sy,e_{X_f})<\delta$ for all $s\in \Gamma\setminus W_3F_1$
which implies $|x_s-y_s|<\var_1$ for all $s\in F_1^{-1}$
and $|y_s|<\var_1$ for all $s\in \Gamma\setminus (W_3F_1)^{-1}$
by the choice of $\delta.$

Take $F=W_1W_2W_3^{-1}W(W_1W_2W_3^{-1}W)^{-1}.$\\

For any finite collection of finite subsets $F_1,F_2,\dots,F_m$ of $\Gamma,$ and
subgroup $\Gamma^{\prime}$ satisfy Condition \ref{spec1} and
\ref{spec2}. By the form of $F,$ we will rewrite the conditions
as
\begin{align}
F_i^{-1}W_1W_2W_3^{-1}W\cap F_i^{-1}W_1W_2W_3^{-1}W = \emptyset \
1\leq i\neq j\leq m,\label{e5}
\end{align}
and
\begin{align}
F_iW_1W_2W_3^{-1}W\cap (\Gamma^{\prime}\{e_\Gamma\})F_j^{-1}W_1W_2W_3^{-1}W=\emptyset
\ \ 1\leq i,j\leq m.\label{e6}
\end{align}
For any collection of points $x^1,x^2,\dots,x^m$  in $X_f,$ we can
pick $y^1,y^2,\dots,y^m\in \Delta(X_f)$ with
$$\max_{s\in (W_1W_2)^{-1}F_i}\rho(sx^i,sy^i)\leq \delta
\ \text{and}\ \sup_{s\in \Gamma\setminus (W_3(W_1W_2)^{-1}F_i)}\rho(sy^i,e_{X_f})\leq \delta$$
which implies
\begin{align}
\max_{s\in F_i^{-1}W_1W_2}|x^i_s-y^i_s|\leq \var_1
\ \text{and} \ \sup_{s\in \Gamma\setminus (F_i^{-1}W_1W_2W_3^{-1})}|y^i_s|\leq \var_1 \label{e7}
\end{align}
for $i=1,2,\dots,m.$
Take $\tilde{y}^i\in [-\frac{1}{2},\frac{1}{2}]^{\Gamma}$
and $\tilde{x}^i\in [-1,1]^{\Gamma}$ with $P(\tilde{y}^i)=y^i$
and $P(\tilde{x}^i)=x^i$ respectively such that
$$|\tilde{x}^i_s-\tilde{y}^i_s|\leq \var_1 \ \forall s\in F_i^{-1}W_1W_2
\ \text{and} \ |\tilde{y}^i_s|\leq \var_1 \quad\forall s\in \Gamma\setminus (F_i^{-1}
W_1W_2W_3^{-1}).$$

For any $s\in \Gamma\setminus (F_i^{-1}W_1W_2W_3^{-1}W),$ one has
$$|(r_{f}\tilde{y}^i)_s|
=|\sum\limits_{t\in \Gamma}f_{t}\tilde{y}_{st}|
=|\sum\limits_{t\in W^{-1}}f_{t}\tilde{y}_{st}|
\leq \|f\|_1\sup\limits_{s\in \Gamma\setminus (W_1W_2W_3^{-1})}|\tilde{y}^i|<\|f\|_1\cdot \var_1<1.$$
We note that
$r_{f}\tilde{y}^{i}\in l^{\infty}(\Gamma, \mathbb{Z}),$
thus $support(r_{f}\tilde{y}^i)\subset F_i^{-1}W_1W_2W_3^{-1}W.$
By (\ref{e5}) and (\ref{e6}), we get the elements $sr_{f}\tilde{y}^{i}$
in $l^{\infty}(\Gamma,\mathbb{Z})$ for $s\in \Gamma^{\prime}$ and
$1\leq i\leq m$ have pairwise disjoint supports.
Denote by $\tilde{z}=\sum_{s\in \Gamma^{\prime}}\sum_{i=1}^{m}sr_{f}\tilde{y}^i.$ By the definition, one
gets $\tilde{z}\in l^{\infty}(\Gamma,\mathbb{Z})$ and
$$\|\tilde{z}\|_{\infty}=
\max\limits_{1\leq i\leq m}\|r_{f}\tilde{y}^{i}\|
\leq \|f\|_1\max\limits_{1\leq i\leq m}\|\tilde{y}^{i}\|_{\infty}
\leq \frac{\|f\|_1}{2}.$$
Set $\tilde{y}=r_{f^{-1}}(\tilde{z})$ and
$y=P(\tilde{y})=\xi(\tilde{z}).$\\
{\bf Claim3:} y satisfies the conditions in specification property.
\\By the property of $\xi,$ one gets $y\in X_f.$ For each $s\in \Gamma^{\prime},$ we have $s\tilde{z}=z$ and hence $sy=y.$
For $x^{\prime},y^{\prime}\in l^{\infty}(\Gamma, \mathbb{R})$
satisfying $\|x^{\prime}\|_{\infty},\|y^{\prime}\|_{\infty}\leq \|f\|_1$ and $x^{\prime},y^{\prime}$ are equal on $sW_2$ for some $s\in \Gamma$  then
\begin{align}
|(r_{f^{-1}}x^{\prime})_s-(r_{f^{-1}}y^{\prime})_s|&=
|\sum_{t\in \Gamma}(f^{-1}_{t}x^{\prime}_{st}-f^{-1}_{t}y^{\prime}_{st})|
=|\sum_{t\in \Gamma\setminus W_2}(f^{-1}_{t}x^{\prime}_{st}-f^{-1}_{t}y^{\prime}_{st})|\nonumber\\
&\leq \frac{\var_1}{2\|f\|_1}\cdot 2\|f\|_1 = \var_1 \label{e8}
\end{align}

\noindent For any $t\in W_1$ and $s\in F_i,$
\begin{align}
|(sx^i)_t-(sy)_{t}|
&=|x^i_{s^{-1}t}-y^i_{s^{-1}t}+y^i_{s^{-1}t}-y_{s^{-1}t}|
\nonumber\\
&\leq |x^i_{s^{-1}t}-y^i_{s^{-1}t}|+|y^i_{s^{-1}t}-y_{s^{-1}t}|\nonumber\\
&\leq 2\var_1,
\end{align}
for $i=1,2,\dots,m$
and the last inequality is getting from (\ref{e7}) and (\ref{e8}).\\
By the choice of $W_1,$ we finish the proof.\qed

\noindent Combined by Theorem \ref{main1} and Theorem \ref{main2},
we can easily get the following corollary.
\begin{Cor}
Let $\Gamma$ be a countable discrete residually finite amenable
group and $f\in \mathbb{Z}\Gamma$ invertible in $l^{\infty}(\Gamma,\mathbb{R}).$ Then for the system $(X_f,\alpha_f),$
$\mcl{M}_{P}(X_f)$ is dense in $\mcl{M}(X_f,\alpha_f).$
\end{Cor}
{\bf Acknowledgement:} We thank Hanfeng Li for all the help  while the
author is visiting SUNY at Buffalo and suggestions in writing this paper.

\noindent \author{Xiankun Ren}, School of Mathematical Sciences,  Peking
University, Beijing 100871, China,and Department of Mathematics, SUNY at
Buffalo, Buffalo, NY 14260-2900, U.S.A.
\\ \email{renxiankun@pku.edu.cn}\\
\email{xiankunr@buffalo.edu}
\end{document}